\documentclass[a4paper]{amsart}
\usepackage{verbatim}
\usepackage{graphicx}
\usepackage{amsthm,amsmath,amsfonts,amssymb}

\newcommand{\cal}{\mathcal}
\newcommand{\cN}{\mathcal{N}}
\newcommand{\cM}{\mathcal{M}}

\newcommand{\dist}{{\rm dist}\,}
\newcommand{\rd}{{\mathbb R}^d}

\newcommand{\R}{{\mathbb R}}
\newcommand{\C}{{\mathbb C}}
\newcommand{\N}{{\mathbb N}}

\newcommand{\diam}{{\rm diam}\,}

\newcommand{\Reg}{{\rm Reg}}
\newcommand{\Ha}{{\cal H}}

\newcommand{\var}{\textnormal{var}} 
\renewcommand{\S}{\mathbf{S}}
\renewcommand{\Re}{\mathrm{Re}}
\newcommand{\bd}[1]{\partial #1} 
\newcommand{\eps}{\varepsilon}
\newcommand{\ind}[1]{\mathbf{1}_{#1}}

\newcommand{\mydot}{\,\cdot\,}

\newcommand{\wlim}[1]{\underset{#1}{\rm{wlim}}\,}
\newcommand{\esswlim}[1]{\underset{#1}{\rm{esswlim}}\,}
\newcommand{\esslim}[1]{\underset{#1}{\rm{esslim}}\,}

\newtheorem{thm}{Theorem}
\newtheorem{prop}[thm]{Proposition}
\newtheorem{lem}[thm]{Lemma}
\newtheorem{cor}[thm]{Corollary}

\theoremstyle{definition}
\newtheorem{ex}[thm]{Example}
\newtheorem{rem}[thm]{Remark}
\numberwithin{equation}{section} \numberwithin{thm}{section}
\begin{document}
\title{Curvature bounds for neighborhoods of self-similar sets}
\author{Steffen Winter}
\address{Karlsruhe Institute of Technology, Department of Mathematics, 76133 Karlsruhe, Germany}

\date{\today}
\subjclass[2000]{Primary 28A75, 28A80; Secondary 28A78, 53C65}
\keywords{self-similar set, parallel set, curvature measures, Minkowski content, Minkowski dimension}
\begin{abstract}
In some recent work, fractal curvatures $C^f_k(F)$ and fractal curvature measures $C^f_k(F,\mydot)$, $k= 0,\ldots,$ $d$, have been determined for all self-similar sets $F$ in $\rd$, for which the parallel neighborhoods satisfy a certain regularity condition and a certain rather technical curvature bound. The regularity condition is conjectured to be always satisfied, while the curvature bound has recently been shown to fail in some concrete examples. As a step towards a better understanding of its meaning, we discuss several equivalent formulations of the curvature bound condition and also a very natural technically simpler condition which turns out to be stronger. These reformulations show that the validity this condition does not depend on the choice of the open set and the constant $R$ appearing in the condition and allow to discuss some concrete examples of self-similar sets. In particular, it is shown that the class of sets satisfying the curvature bound condition is strictly larger than the class of sets satisfying the assumption of polyconvexity used in earlier results. 
\end{abstract}
\maketitle

\section{Introduction}

Total curvatures and curvature measures are well known for certain classes of sets in Euclidean space $\R^d$ including convex bodies, differentiable submanifolds with boundary, sets with positive reach and certain unions of such sets. 
In convex geometry, total curvatures are better known as intrinsic volumes or Min\-kow\-ski functionals and in differential geometry as integrals of mean curvatures. Curvature measures were introduced by Federer \cite{F59} for sets with positive reach and have later been extended in various directions, see e.g.\ \cite{CMS84, BK00,Za87, Za90}. 

In some recent work fractal counterparts -- so called \emph{fractal curvatures} and \emph{fractal curvature measures} -- have been introduced for certain classes of self-similar fractals, cf.~\cite{winter, Za09, wz10}, based on the following ideas: A compact (fractal) set $K\subset\R^d$ is well approximated by its $\eps$-parallel sets
$$K_\eps:=\{x\in\rd: \dist(x,K)\le\eps\}\, $$
as $\eps$ tends to 0 (in the sense of Hausdorff metric) and for sufficiently regular sets $K$ the curvature measures behave nicely under such approximation. Also for singular sets $K$, the parallel sets are often regular enough to admit curvatures measures $C_k(K_\eps,\cdot)$. In this case  fractal curvatures are explained as suitably scaled limits of the total curvatures $C_k(K_\eps):= C_k(K_\eps, \R^d)$ and fractal curvature measures as the corresponding weak limits of the curvature measures, as $\eps$ tends to zero. 

The focus of recent work has been to establish the existence of these limits for certain classes of (self-similar) sets. 
In \cite{winter}, where these concepts were introduced, the existence of fractal curvatures and fractal curvature measures 
was established for self-similar sets with polyconvex parallel sets. This polyconvexity assumption has been dropped in \cite{Za09} for the fractal curvatures and in \cite{wz10} for fractal curvature measures. In the former paper, also random self-similar sets are treated. In these papers the polyconvexity is replaced by two technical conditions. One is a regularity condition on the parallel sets, which ensures that the curvature measures of the $\eps$-parallel sets are well defined for almost all $\eps$ (see condition RC below). This condition is certainly weaker than the polyconvexity assumption as it is known to be satisfied for all sets in $\R^d$, $d\le 3$. Moreover, it is conjectured to be always satisfied for self-similar sets satisfying the open set condition, see the discussion below.  
The second condition is a bound on the curvature of $F_\eps$ near certain intersections of the cylinder sets of $F$, cf.~condition CBC below. This curvature bound condition is not very well understood. As it involves cylinder sets of $F$ of all scales, it is rather difficult to verify in concrete examples. But it is believed to be satisfied for most self-similar fractals. Very recently, some self-similar sets for which CBC does not hold have been discovered independently by Andreas Wust and Jan Rataj, giving thus a negative answer to the question whether CBC holds for all self-similar sets, see Example~\ref{ex:wust} below.\\ 

In this note we discuss the curvature bound condition in some greater detail. We will give several equivalent reformulations of this condition. In particular, this will allow to show that the validity of CBC does neither depend on the choice of the open set $O$ (a feasible set for the strong open set condition) nor on the choice of the constant $R$, which appear both in the original formulation of CBC. This removes some arbitrariness from the condition. The condition can not be weakened or strenghtened by making a different choice of $O$ or $R$. Some of the reformulations of CBC are also helpful when discussing examples, as they are easier to verify.
We also discuss a technically much simpler curvature bound which involves only first level cylinder sets. This bound was a natural candidate for an equivalent reformulation of CBC but turned out to be slightly stronger, hence the term strong curvature bound condition (SCBC) used in the sequel. This condition is interesting in practice, as it implies CBC and is much easier to verify. On the other hand, it enlightens to some extend, why some knowledge of the fine structure provided by CBC is necessary. In general, one needs to know something about the intersections of cylinder sets at all scales. For certain 'simple' fractals, knowledge of the first level suffices. Here `simple' roughly means that the intersections of the parallel sets of first level cylinder sets have no `fractal' structure. 
We illustrate the results by verifying CBC for the Koch curve (using SCBC) and for some other set for which SCBC fails.
These two examples are sets, which do not have polyconvex parallel sets but for which CBC holds. They show in particular, that the class of sets covered by the results in \cite{Za09} and \cite{wz10} is strictly larger than the class of sets with polyconvex parallel sets considered in \cite{winter}.

The paper is organized as follows. In the next section, we collect some well known facts about curvature measures required later on. In Section~\ref{sec:fcm}, we recall the curvature bound condition and the results from \cite{Za09} and \cite{wz10} on the existence of fractal curvatures and fractal curvature measures for self-similar sets. Finally, in Sections~\ref{sec:CBC} and~\ref{sec:SCBC}, the main results are presented. Several equivalent reformulations of CBC and their consequences are discussed in Sections~\ref{sec:CBC}, while SCBC is the subject of interest in Section~\ref{sec:SCBC}. In these sections also the examples are found.
 
\section{Curvature measures} \label{sec:cm}
We denote the {\it closure of the complement} of a compact set $K$ by $\widetilde{K}$.
A distance $\eps\ge 0$ is called {\it regular} for the set $K$ if $\widetilde{K_\eps}$ has positive reach in the sense of Federer \cite{F59} and the boundary $\partial K_\eps$ is a Lipschitz manifold. In view of Fu \cite{Fu85}, in space dimensions $d\le 3$ this is fulfilled for Lebesgue almost all $\eps$. (For general $d$, a sufficient condition for this property is that $\eps$ is a regular value of the distance function of $K$ in the sense of Morse theory, cf.~\cite{Fu85}.) For regular $\eps$ the {\it Lipschitz-Killing curvature measures} of order $k$ are determined by means of Federer's versions for sets of positive reach:
\begin{equation}
C_k(K_\eps,\mydot):=(-1)^{d-1-k}C_k(\widetilde{K_\eps},\mydot)\, ,~~k=0,\ldots,d-1\, ,
\end{equation}
where the surface area ($k=d-1$) is included and the volume measure $C_d(K_\eps,\mydot):=\lambda_d(K_\eps\cap\mydot)$ is added for completeness. For more details and some background on singular curvature theory for fractals we refer to \cite{winter, Za09}.\\
The {\it total curvatures} of $K_\eps$ are denoted by
\begin{equation}
C_k(K_\eps):=C_k(K_\eps,\rd)\, ,~~ k=0,\ldots ,d\, .
\end{equation}
We recall now the main properties of curvature measures required for our purposes:
By an associated Gauss-Bonnet theorem the {\it Gauss curvature}
$C_0(K_\eps)$ coincides with the {\it Euler-Poincar\'{e} characteristic} $\chi(K_\eps)$.\\
The curvature measures are {\it motion invariant}, i.e.,
\begin{equation} \label{eqn:motion}
C_k(g(K_\eps),g(\mydot))=C_k(K_\eps,\mydot)~~\mbox{for any Euclidean motion}~g\, ,
\end{equation} 
the $k$-th measure is {\it homogeneous of degree} $k$, i.e.,
\begin{equation} \label{eqn:scale}
C_k(\lambda K_\eps,\lambda \mydot)=\lambda^k\, C_k(K_\eps,\mydot)\, ,~~\lambda>0\, ,
\end{equation}
and they are {\it locally determined}, i.e.,
\begin{equation} \label{eqn:loc}
C_k(K_\eps,\mydot\cap G)=C_k(K'_{\eps '},\mydot\cap G)
\end{equation}
for any open set $G\subset\rd$ such that $K_\eps\cap G=K'_{\eps '}\cap G$, where $K_\eps$ and $K'_{\eps '}$ are
both parallel sets such that the closures of their complements have positive reach.

Finally, for sufficiently large distances the parallel sets are always regular and the curvature measures may be estimated by those of a ball of almost the same size: For any compact set $K\subset\R^d$ and any $\eps\ge R>\sqrt{2}\,\diam K$ we have
\begin{align} \label{eq:R-big}
 C^\var_k(K_\eps)&\le c_k(K,R)\, \eps^k\,,
\end{align}
for some constant $c_k(K,R)$ independent of $\eps$, see \cite[Thm.~4.1]{Za09}). Here $C^\var_k(K_\eps,\mydot)$ denotes the total variation measure of $C_k(K_r,\mydot)$ and $C^\var_k(K_r):=C^\var_k(K_r,\R^d)$ its total mass.\\


\section{Existence of fractal curvatures and fractal curvature measures} \label{sec:fcm}

In this section, we briefly recall the results on fractal curvatures and fractal curvature measures obtained in \cite{Za09, wz10}. For this purpose, we  recall first some concepts related to self-similar sets and give a precise formulation of the regularity condition and the curvature bound condition. 
   
For $N\in\N$ and $i=1,\ldots,N$, let $S_i:\R^d\to\R^d$ be a contracting similarity with contraction ratio $0<r_i<1$.
Let $F\subset\R^d$ be the \emph{self-similar set} generated by the function system $\{S_1,\ldots, S_N\}$. That is, $F$ is the unique nonempty, compact set invariant under the set mapping $\S (\mydot):=\bigcup_i S_i(\mydot)$, cf.~\cite{Hut81}. The set $F$ (or, more precisely, the system $\{S_1,\ldots,S_N\}$) is said to satisfy the \emph{open set
condition} (OSC) if there exists a non-empty, open and bounded
subset $O$ of $\R^d$
 such that
 $$
 \bigcup _i S_i O \subseteq O \quad \text{ and } \quad S_i O \cap
S_j O=\emptyset  \text{ for } i\neq j\,.
$$
The \emph{strong open set condition} (SOSC) holds for $F$ (or $\{S_1,\ldots,S_N\}$), if there exist a set $O$ as in the OSC which additionally satisfies $O\cap F\neq\emptyset$.
It is well known that in $\R^d$ OSC and SOSC are equivalent, cf.~\cite{schief}, i.e., for $F$ satisfying OSC, there exists always such a set $O$ with $O\cap F\neq\emptyset$.

The unique solution $s=D$ of the equation $\sum_{i=1}^N r_i^s=1$ is called the \emph{similarity dimension} of $F$.  It is well known that for self-similar sets $F$ satisfying OSC, $D$ coincides with Minkowski and Hausdorff dimension of $F$.
Further, a self-similar set $F$ is called \emph{arithmetic} (or \emph{lattice}), if there exists some number $h>0$ such
that $-\ln r_i \in h\mathbb{Z}$ for $i=1,\ldots, N$, i.e.\ if $\{-\ln r_1,\ldots,-\ln r_N\}$ generates a discrete subgroup of $\R$.
Otherwise $F$ is called \emph{non-arithmetic} (or \emph{non-lattice}).

Let $\Sigma^*:=\bigcup_{j=0}^\infty\{1,\ldots,N\}^j$ be the set of all finite words over the alphabet $\{1,\ldots,N\}$ including the emtpy word. For $\omega=\omega_1\ldots \omega_n\in\Sigma^*$ we denote by $|\omega|$ the \emph{length of $\omega$} (i.e., $|\omega|=n$) and by $\omega|k:=\omega_1\ldots \omega_k$ the subword of the first $k\le n$ letters. We abbreviate $r_\omega:=r_{\omega_1}\ldots r_{\omega_n}$ and  $S_\omega:=S_{\omega_1}\circ\ldots\circ S_{\omega_n}$. 

Throughout we assume that $F$ is a self-similar set in $\R^d$ satisfying OSC and that $D$ denotes its similarity dimension. Furthermore, we assume that the following regularity condition is satisfied:
\begin{itemize}
\item[(RC)\quad] Almost all $\eps\in(0, \sqrt{2}\diam(F))$ are regular for $F$. 
\end{itemize}
That is, the set $\cM$ of irregular values is a Lebesgue null set. This condition is always satisfied for subsets of $\R^d$, $d\le 3$, cf.\  Section~\ref{sec:cm}. For self-similar sets in $\R^d$ satisfying OSC, it is conjectured to be true for all $d$. Note that there are no irregular values $\eps\ge\sqrt{2}\diam(F)$, cf.\ for instance~\cite[Theorem 4.1]{Za09}.

In order to be able to formulate the curvature bound condition (CBC), we need to fix some constant $R=R(F)$ for $F$ such that
\begin{equation}\label{eq:R}
R>\sqrt{2}\, \diam F
\end{equation}
(to be able to apply \eqref{eq:R-big}) and some open set $O=O(F)$ satisfying SOSC.  Note that the choice of $R$ and $O$ are otherwise completely arbitrary. 
For $0<\eps\le R$, let $\Sigma(\eps)$ be the family of all finite words $\omega=\omega_1\ldots \omega_n\in\Sigma^*$ such that
\begin{equation} \label{Sigma}
	Rr_\omega< \eps \le Rr_{\omega||\omega|-1},
\end{equation}
and let
\begin{equation} \label{Sigma_b}
	\Sigma_b(\eps):=\{\omega\in\Sigma(\eps): (S_\omega F)_\eps\cap (\S O)^c_\eps \neq \emptyset\}.
\end{equation}
The words $\omega $ in $\Sigma(\eps)$ describe those cylinder sets $S_\omega  F$ which are approximately of size $\eps$ and the words in $\Sigma_b(\eps)$ only those which are also $2\eps$-close to the boundary of the set $\S O$, the first iterate of the set $O$ under the set mapping $\S =\bigcup_{i=1}^N S_i$. Note that the family $\{S_\omega F: \omega\in\Sigma(\eps)\}$ is a covering of $F$ for each $\eps$, which is optimal in that none of the sets can be removed. It is an easy consequence of the equation $\sum_{i=1}^N r_i^D=1$ that, for each $\eps\in(0,R]$,
\begin{equation} \label{Sigma-eqn}
	\sum_{\omega\in\Sigma(\eps)} r_\omega^D=1\,.
\end{equation}
In \cite{wz10}, the curvature bound condition is formulated as follows:
\begin{itemize}
\item[(CBC)]   There is a constant $c_k$ such that for almost all $\eps\in(0,R)$ and all $\sigma\in \Sigma_b(\eps)$
\begin{equation*}
C_k^\var\left(F_\eps, \bd (S_\sigma F)_\eps \cap \bd \bigcup_{\sigma'\in\Sigma(\eps)\setminus\{\sigma\}}(S_{\sigma'}F)_\eps\right)\le c_k\eps^k.
\end{equation*}
\end{itemize}

The following result on the limiting behaviour of the total curvatures was obtained in \cite{Za09}. We restrict our attention to the deterministic case.
Set
\begin{equation} \label{Rk-def}
R_k(\eps) := C_k(F_\eps)-\sum_{i=1}^N \ind{(0,r_i]}(\eps) C_k((S_i F)_\eps),\quad \eps>0.
\end{equation}

\begin{thm} \cite[Theorem~2.3.8 and Corollary~2.3.9]{Za09} \label{thm:global}
Let $k\in\{0,1,\ldots,d\}$ and $F$ be a self-similar set in $\R^d$, $d\ge1$, satisfying OSC. If $k\le d-2$, assume additionally that RC and CBC hold.
Then
\begin{equation}\label{eq:global:lim}
C_k^f(F):=\lim_{\delta\to 0} \frac 1{|\ln\delta|}\int_\delta^1\eps^{D-k} C_k(F_\eps)\frac{d\eps}\eps =\frac 1{\eta} \int_0^R r^{D-k-1}R_k(r) dr,
\end{equation}
where $\eta=-\sum_{i=1}^N r_i^D \ln r_i$.
Moreover, if $F$ is non-arithmetic,
then
\begin{equation}\label{eq:global:esslim}
\esslim{\eps\to 0} \eps^{D-k} C_k(F_\eps)=C_k^f(F).
\end{equation}
\end{thm}

The numbers $C_k^f(F)$ are refered to as the \emph{fractal curvatures} of the set $F$. Formula \eqref{eq:global:lim} in Theorem~\ref{thm:global} should in particular be understood to imply that the integral on the right hand side exists and thus the fractal curvatures are finite.
For $k=d$, the limits in \eqref{eq:global:lim} and \eqref{eq:global:esslim} specialize to the average Minkowski content and the Minkowski content, respectively, and the result is due to Lapidus and Pomerance \cite{LapPo1}, Falconer \cite{fal93} (for $d=1$) and Gatzouras \cite{gatzouras} (for general $d$). The case $k=d-1$ is treated in \cite{rw09}. In both cases the essential limits can be replaced by limits and  the limits are always positive. Recall that for $d\le 3$ RC is known to be satisfied. For the special case of polyconvex parallel sets, where the conditions RC and CBC are not needed, see \cite{winter}.

It is shown in~\cite{wz10}, that under the hypotheses of Theorem~\ref{thm:global} also \emph{fractal curvature measures} exist.

\begin{thm}\label{thm:local} \cite[Theorem~2.3]{wz10}
Let $k\in\{0,1,\ldots,d\}$ and $F$ be a self-similar set in $\R^d$, $d\ge1$, satisfying OSC. If $k\le d-2$, assume additionally that RC and CBC hold.
Then
\begin{equation}\label{eq:mainthm:wlim}
C_k^f(F,\mydot):=\wlim{\eps\to 0}\frac{1}{|\ln\eps|}\int_\eps^1 \tilde{\eps}^{D-k} C_k(F_{\tilde{\eps}},\mydot)\frac{d\tilde{\eps}}{\tilde{\eps}}=C_k^f(F)\mu_F,
\end{equation}
where
$\mu_F$ is the normalized $D$-dimensional Hausdorff measure on $F$. Moreover, if $F$ is non-arithmetic,
then also
the essential weak limit 
$\esswlim{\eps\to 0}\eps^{D-k} C_k(F_\eps,\mydot)$ 
exists and equals $C_k^f(F,\mydot)$.
\end{thm}

\section{Equivalent reformulations of CBC} \label{sec:CBC}

We give some alternative equivalent formulations of CBC with the intension to clarify the meaning of this condition and also to simplify its verification in concrete examples. 

Throughout we assume that $k\in\{0,\ldots,d-2\}$ (since for $k\in\{d-1,d\}$ CBC is not needed) and that $F$ is a self-similar set in $\R^d$ satisfying OSC and RC. 
The first equivalent reformulation of CBC is rather obvious and has been mentioned in \cite[cf.~Remark~2.4]{wz10} already: 
The boundary signs in CBC can be omitted. It paves the road for further reformulations. For $\eps\in(0,R)$ and $\sigma\in\Sigma(\eps)$, let
\begin{equation}\label{eq:A-sigma}
A^{\sigma,\eps}:=\bigcup_{\sigma' \in\Sigma(\eps)\setminus\{\sigma\}}(S_{\sigma'}F)_\eps.
\end{equation}

\begin{prop} \label{CBC1}
The following condition is equivalent to CBC:\vspace{1mm}\\
{\bf (CBC1)} There is a constant $c_k$ and a null set $\cN\subset(0,R)$ such that for all $\eps\in(0,R)\setminus\cN$ and all $\sigma\in \Sigma_b(\eps)$
\begin{equation}\label{eq:equi-cond1}
C_k^\var\left(F_\eps, (S_\sigma F)_\eps \cap A^{\sigma,\eps}\right)\le c_k\eps^k.
\end{equation}
\end{prop}
\begin{proof}
The assertion follows from the set equality 
$$
(S_\sigma F)_\eps\cap A^{\sigma,\eps}\cap \partial F_\eps = \partial (S_\sigma F)_\eps\cap \partial A^{\sigma,\eps}\cap \partial F_\eps
$$
and the fact that the curvature measure $C_k(F_\eps,\mydot)$ is concentrated on the boundary of $F_\eps$, see also \cite[Remark 2.4]{wz10}.
\end{proof}

\begin{rem} \label{rem:cN}
Without loss of generality, we can assume that the set $\cN$ in CBC1 has the following additional properties:
\begin{equation} \label{cN-props}
\cM\subseteq\cN \quad \text{ and } \quad r_\sigma\cN\subset\cN \text{ for all }\sigma\in\Sigma^*,
\end{equation}
where $\cM$ is the (Lebesgue null) set of exceptions in RC.
Indeed, the existence of a null set $\cN$ satisfying these additional conditions clearly implies the existence of a null set at all satisfying CBC1. Conversely, if CBC1 holds with an arbitrary null set $\cN$ of exceptions, then it also holds with the larger null set $\cN^*:=\bigcup_{\sigma\in\Sigma^*} r_\sigma(\cM\cup\cN)\subset(0,R)$ of exceptions, which has both of the above properties. In the sequel we will always assume that the set $\cN$ of exceptions has these two additional properties.
\end{rem}

For the proof of the next reformulation we require the following estimate, which is proved in \cite{wz10}. Recall the definition of the set $A^{\sigma,\eps}$ from \eqref{eq:A-sigma}.
\begin{lem}\cite[Lemma~3.2]{wz10} \label{curv-est}
Let $k\in\{0,\ldots,d-2\}$ and let $F$ be a self-similar set in $\R^d$ satisfying OSC, RC and CBC.
Then there is a constant $c>0$ such that, for all $\eps\in(0,R)\setminus\cN$ 
and all $\sigma\in\Sigma(\eps)$,
\begin{align}\label{eqn:curvB2}
C_k^\var(F_\eps,(S_\sigma F)_\eps\cap A^{\sigma,\eps})&\le c \eps^k\,.
\end{align}
\end{lem} 

In the following reformulation of CBC we shift the parameter $r$ in the families $\Sigma(r)$ in order to be able to work with larger cylinder sets compared to the parallel width $\eps$. Condition CBC2 below roughly means that one can work with cylinder sets of diameter $\lambda \eps$, $\lambda\ge 1$.  Practically, this allows to reduce the number of mutual intersections between the cylinder sets. It also enables us to show that the validity of CBC for a given self-similar set $F$ does not depend on the choice of the constant $R$.

\begin{thm} \label{thm:CBC2}
Let $k\in\{0,\ldots,d-2\}$ and let $F$ be a self-similar set in $\R^d$ satisfying OSC and RC. Let $\lambda\ge 1$.
Then the following condition is equivalent to CBC:\vspace{1mm}\\ 
{\bf (CBC2)} 
There exist $b_k=b_k(\lambda)>0$ and a null set $\cN$ such that for all $\eps\in(0,R/\lambda)\setminus\cN$ and all $\omega\in \Sigma_b(\lambda\eps)$
\begin{equation*}
C_k^\var\left(F_\eps, (S_\omega F)_\eps \cap \bigcup_{\omega'\in\Sigma (\lambda\eps)\setminus\{\omega\}}(S_{\omega'}F)_\eps\right)\le b_k\eps^k\,,
\end{equation*}
and such that for all $\eps\in[R/\lambda,R)\setminus\cN$ 
\begin{equation*}
C_k^\var\left(F_\eps\right)\le b_k\eps^k.
\end{equation*}
\end{thm}

\begin{proof} 
For $\lambda=1$, CBC1 and CBC2 are obviously equivalent, since the first inequality in CBC2 reduces to CBC1 in this case and the range of the second one is the empty set. (The second inequality should be viewed as an extension of the range of \eqref{eq:R-big} to the interval $[R/\lambda,R)$.)

So fix some $\lambda>1$.  
We first show that CBC1 implies CBC2. For $\omega=\omega_1\ldots \omega_m\in \Sigma (\lambda\eps)$, let
$$
\Sigma^\omega(\eps):=\{\sigma\in\Sigma(\eps): \sigma_i=\omega_i \text{ for } i=1,\ldots,m\}.
$$
Observe that the cardinality of the sets $\Sigma^\omega(\eps)$ is bounded by a constant (independent of $\eps\in(0,R)$ and $\omega\in\Sigma (\lambda\eps)$).
Indeed, each $\sigma\in\Sigma^\omega(\eps)$ is of the form $\sigma=\omega \tilde{\sigma}$ with $\tilde{\sigma}\in\Sigma(\eps/r_\omega)$. Hence
\begin{equation*}
\#\Sigma^\omega(\eps)\le\#\Sigma(\eps/r_\omega)\le \#\Sigma(R\lambda^{-1}) =:\hat{c}\,,
\end{equation*}
where the last inequality is due to the relation $\eps/r_\omega>\lambda^{-1}R$ (since $\omega\in\Sigma(\lambda\eps)$) and the monotonicity of $\#\Sigma(\mydot)$.
Since $(S_\omega F)_\eps=\bigcup_{\sigma\in\Sigma^\omega(\eps)} (S_\sigma F)_\eps$, we have for each $\eps\in(0,\lambda^{-1}R)\setminus\cN$,
\begin{align*}
C_k^\var &\left(F_\eps, (S_\omega F)_\eps \cap \bigcup_{\omega'\in\Sigma (\lambda\eps)\setminus\{\omega\}}(S_{\omega'}F)_\eps\right)\\
&= C_k^\var\left(F_\eps, \bigcup_{\sigma\in\Sigma^\omega(\eps)} (S_\sigma F)_\eps \cap \bigcup_{\omega'\in\Sigma (\lambda\eps)\setminus\{\omega\}}(S_{\omega'}F)_\eps\right)\\
&\le \sum_{\sigma\in\Sigma^\omega(\eps)} C_k^\var\left(F_\eps, (S_\sigma F)_\eps \cap \bigcup_{\omega'\in\Sigma (\lambda\eps)\setminus\{\omega\}} 
(S_{\omega'}F)_\eps \right)\\
&\le \sum_{\sigma\in\Sigma^\omega(\eps)} C_k^\var\left(F_\eps, (S_\sigma F)_\eps \cap \bigcup_{\sigma'\in\Sigma(\eps)\setminus\{\sigma\}} (S_{\sigma'}F)_\eps\right)\\
&= \sum_{\sigma\in\Sigma^\omega(\eps)} C_k^\var\left(F_\eps, (S_\sigma F)_\eps \cap A^{\sigma,\eps}\right),
\end{align*}
where the last inequality is due to the set inclusion
$$\bigcup_{\omega'\in\Sigma (\lambda\eps)\setminus\{\omega\}}(S_{\omega'}F)_\eps\subseteq \bigcup_{\sigma'\in\Sigma(\eps)\setminus\{\sigma\}} (S_{\sigma'}F)_\eps
$$
and the last equality to \eqref{eq:A-sigma}.
Now, since CBC1 is assumed to hold (which is equivalent to CBC by Proposition~\ref{CBC1}), we can apply Lemma~\ref{curv-est} and obtain
that each of the terms in this sum is bounded from above by $c\eps^k$. Therefore, the whole sum is bounded by $b_k\eps^k$ with $b_k:=\hat{c} c$, showing the first inequality of CBC2. The second inequality follows immediatly from \cite[Corollary 4.1]{wz10}, which states that CBC implies the uniform boundedness of $\eps\mapsto C_k^{\var}(F_\eps)$ on compact intervals $[a,b]\subset(0,\infty)$. 
For the convenience of the reader, we provide the following direct alternative proof of the second inequality: 
Observe that for $\eps\in(0,R)\setminus\cN$
\begin{align*}
C_k^\var (F_\eps)&= C_k^\var\left(F_\eps, \bigcup_{\sigma\in\Sigma(\eps)} (S_\sigma F)_\eps\right)\\
&\le \sum_{\sigma\in\Sigma(\eps)} C_k^\var\left(F_\eps, (S_\sigma F)_\eps \right)\\
&\le \sum_{\sigma\in\Sigma(\eps)} C_k^\var\left(F_\eps, (S_\sigma F)_\eps \cap A^{\sigma,\eps}\right) + C_k^\var\left(F_\eps, (S_\sigma F)_\eps \cap (A^{\sigma,\eps})^c\right).
\end{align*}
By Lemma~\ref{curv-est}, for each $\sigma\in\Sigma(\eps)$, the first term in this sum is bounded by $c\eps^k$. For the second term, we have $F_\eps \cap (A^{\sigma,\eps})^c=(S_\sigma F)_\eps \cap (A^{\sigma,\eps})^c$ and so, by the locality property \eqref{eqn:loc}, 
\begin{align*}
C_k^\var\left(F_\eps, (S_\sigma F)_\eps \cap (A^{\sigma,\eps})^c\right)&=C_k^\var\left((S_\sigma F)_\eps, (S_\sigma F)_\eps \cap (A^{\sigma,\eps})^c\right)\\
&\le C_k^\var\left((S_\sigma F)_\eps\right)=r_\sigma^k C_k^\var\left(F_{\eps/r_\sigma} \right).
 \end{align*}
Since $\sigma\in\Sigma(\eps)$ implies $\frac\eps{r_\sigma}>R$, the last term is bounded by $r_\sigma^k c(F,R) \left(\frac\eps{r_\sigma}\right)^k=c(F,R)\eps^k$, by \eqref{eq:R-big}.
Finally observe that, for $\eps\in[\lambda^{-1}R,R)$, the cardinality of the family $\Sigma(\eps)$ is bounded by the constant $\tilde c:= \#\Sigma(\lambda^{-1}R)$ and thus we conclude that $C_k^\var (F_\eps)$ is bounded by $b_k\eps^k$ (with $b_k=\tilde c(c+c(F,R))$) for $\eps\in[\lambda^{-1}R,R)$ as claimed in the second inequality in CBC2. This completes the proof of the implication CBC1 $\Rightarrow$ CBC2.

For the reverse implication, let first $\eps\in[\lambda^{-1}R,R)\setminus\cN$. Then, by the second inequality in CBC2, we immediatly obtain for each $\sigma\in\Sigma(\eps)$,
\begin{align*}
C_k^\var\left(F_\eps, (S_\sigma F)_\eps \cap A^{\sigma,\eps}\right)&\le C_k^\var\left(F_\eps\right)\le b_k\eps^k\,,
\end{align*}
which verifies the inequality in CBC1 for $\eps\in[\lambda^{-1}R,R)\setminus\cN$. 
Now let $\eps\in(0,\lambda^{-1}R)\setminus\cN$ and $\sigma\in \Sigma(\eps)$. Let $\omega\in\Sigma(\lambda\eps)$ be the unique sequence such that
$\sigma=\omega\tilde\sigma$. In analogy with \eqref{eq:A-sigma}, we set
\begin{equation} \label{eqn:B-omega}
B^{\omega,\eps}:=\bigcup_{\omega'\in\Sigma (\lambda\eps)\setminus\{\omega\}}(S_{\omega'}F)_\eps\,.
\end{equation}
Since $(S_\sigma F)_\eps\subseteq (S_\omega F)_\eps$ and obviously 
$
A^{\sigma,\eps}
\subseteq \R^d = B^{\omega,\eps}\cup (B^{\omega,\eps})^c$ 
we infer, that
\begin{align} \label{eqn:CBC2to1}
C_k^\var\left(F_\eps, (S_\sigma F)_\eps \cap A^{\sigma,\eps}\right)
&\le C_k^\var\left(F_\eps, (S_\omega F)_\eps \cap B^{\omega,\eps}\right)
+ C_k^\var\left(F_\eps, (S_\omega F)_\eps \cap (B^{\omega,\eps})^c\right) 
\end{align}
Now, if we assume $\sigma\in\Sigma_b(\eps)$, then $\omega\in\Sigma_b(\lambda \eps)$. Therefore, by CBC2, the first term in the above expression is bounded by $b_k\eps^k$.
For the second term observe that, by the locality property \eqref{eqn:loc} (since $\eps$ and thus $\eps/r_\omega$ in $(0,R)\setminus\tilde\cN$) in the open set $(B^{\omega,\eps})^c$, we can replace $F_\eps$ by $(S_\omega F)_\eps$. Hence this term is bounded by
\begin{align*}
C_k^\var\left((S_\omega F)_\eps, (S_\omega F)_\eps \cap (B^{\omega,\eps})^c\right)
&\le r_\omega^k C_k^\var\left(F_{\eps/r_\omega}\right)\,.
\end{align*}
Finally, recalling that $w\in\Sigma(\lambda\eps)$ and so $\eps/r_\omega> \lambda^{-1}R$, we conclude from the second inequality in CBC2 that the last expression (and thus the second term in \eqref{eqn:CBC2to1}) is bounded by $b_k \eps^k$. This verifies the inequality in CBC1 for $\eps\in(0,\lambda^{-1}R)\setminus\cN$ and $\sigma\in\Sigma_b(\eps)$ and completes the proof of the implication CBC2 $\Rightarrow$ CBC1.  
\end{proof}

Note that condition CBC2 in Theorem~\ref{thm:CBC2} can equivalently be phrased ``There exists a constant $\lambda\ge 1$, a constant $b_k=b_k(\lambda)$ and \ldots,'' or  ``For all $\lambda\ge 1$, there exists a constant $b_k=b_k(\lambda)$ and \ldots''.  The next statement shows that it is not important how the constant $R$ is chosen. If for a self-similar set, CBC fails to hold  for some $R$, it can not be verified by chosing a different $R$. 

\begin{cor}
CBC is independent of the choice of the constant $R$, i.e., if $R_1$ and $R_2$ are two constants  with $R_i>\sqrt{2}\diam F$, 
then CBC with $R=R_1$ is satisfied if and only if CBC with $R=R_2$ is.
\end{cor}
\begin{proof}
Without loss of generality, we may assume that $R_1>R_2$. Suppose CBC1 holds with $R=R_1$ and let $\lambda:=\frac{R_1}{R_2}>1$. Then, by Theorem~\ref{thm:CBC2}, CBC2 holds with $R=R_1$ and $\lambda=\frac{R_1}{R_2}$. Since $\frac{R_1}{\lambda}=R_2$ and $\Sigma^{(R_1)}(\lambda\eps)=\Sigma^{(R_2)}(\eps)$ (where the superscripts $R_1$ and $R_2$ indicate which $R$ we have to use in the definition of $\Sigma(r)$), we have in particular that for all $\eps\in(0,R_2)\setminus\cN$ and for all $w\in\Sigma_b^{(R_2)}(\eps)$ 
\begin{equation*}
C_k^\var\left(F_\eps, (S_\omega F)_\eps \cap \bigcup_{\omega'\in\Sigma^{(R_2)}(\lambda\eps)\setminus\{\omega\}}(S_{\omega'}F)_\eps\right)\le b_k\eps^k\,.
\end{equation*}
which is just CBC1 with $R=R_2$.

Conversely, if CBC1 with $R=R_2$ holds, then the argument from above shows that the first inequality of CBC2 with $R=R_1$ and $\lambda=\frac{R_1}{R_2}$ also holds. Moreover, by \eqref{eq:R-big}, there exists a constant $c=c(F,R_2)$ such that $C_k^\var(F_\eps)\le c\eps^k$ for $\eps>R_2=\frac{R_1}{\lambda}$, i.e., in particular, for $\eps\in(\frac{R_1}{\lambda}, R_1]$. Hence, the second inequality of CBC2 with $R=R_1$ is also satisfied. Now, again by Theorem~\ref{thm:CBC2}, we infer that CBC1 with $R=R_1$ holds, which completes the proof.   
\end{proof}

Condition CBC3 below shows that if the cylinder sets are chosen large enough (compared to $\eps$), then one can pass over to mutual intersections of pairs of cylinder sets. The proof is based on a lemma in \cite{winter}, which roughly says that a set $(S_\sigma F)_\eps$ from a family $\{(S_\omega F)_\eps:\omega\in\Sigma(\lambda\eps)\}$ does not intersect too many of the other members of this family, provided $\lambda$ is large enough, cf.~\cite[Lemma 5.3.1]{winter}. More precisely, $\lambda$ needs to be larger than $R\rho^{-1}$, where $\rho$ is given as follows: Because of SOSC, there exists a word $u\in\Sigma^*$ such that $S_uF\subset O$ and the compactness of $S_u F$ implies that there is a constant $\alpha>0$ such that each point $x\in S_uF$ has a distance greater than $\alpha$ to $\partial O$, i.e., $ d(x,O^c)>\alpha$. Set $\rho:=r_{\min}\frac\alpha 2$, where $r_{\min}:=\min\{r_i: 1\le i\le N\}$. (Note that $\rho$ depends on the choice of $O$ and the word $u$. Any choice $\rho\le r_{\min}\frac\alpha 2$ is also fine.) Compare also \cite[Section 5.1]{winter}. 





\begin{thm} \label{thm:CBC3}
Let $k\in\{0,\ldots,d-2\}$ and let $F$ be a self-similar set in $\R^d$ satisfying OSC and RC. Let $\lambda\ge \max\{1, R\rho^{-1}\}$. Then the following condition is equivalent to CBC:\vspace{1mm}\\
{\bf (CBC3)} There is a constant $a_k=a_k(\lambda)$ and a null set $\cN$ such that for all $\eps\in(0,R/\lambda)\setminus\cN$, $\omega\in\Sigma_b(\lambda \eps)$ and $\omega'\in \Sigma (\lambda\eps)\setminus\{\omega\}$
\begin{equation*}
C_k^\var\left(F_\eps, (S_\omega F)_\eps \cap (S_{\omega'}F)_\eps\right)\le a_k\eps^k
\end{equation*}
and such that for all $\eps\in[R/\lambda,R)\setminus\cN$ 
\begin{equation*}
C_k^\var\left(F_\eps\right)\le a_k\eps^k.
\end{equation*}
\end{thm}

\begin{proof}
Fix some $\lambda\ge \max\{1, R\rho^{-1}\}$. In view of Theorem~\ref{thm:CBC2}, it suffices to show that CBC3 is equivalent to CBC2 (with the same $\lambda$ and $\cN$).
The implication CBC2 $\Rightarrow$ CBC3 is easy: If $\eps\in(0,R)\setminus\cN$, $\omega\in\Sigma_b(\lambda\eps)$ and $\omega'\in\Sigma(\lambda\eps)$, then
$$
C_k^\var\left(F_\eps, (S_\omega F)_\eps \cap (S_{\omega'}F)_\eps\right)\le C_k^\var\left(F_\eps, (S_\omega F)_\eps \cap \bigcup_{v\in\Sigma (\lambda\eps)\setminus\{\omega\}}(S_{v}F)_\eps\right)
$$
and, by CBC2, the right hand side is bounded by $b_k\eps^k$, verifying the first inequality of CBC3. The second inequalities are obviously equivalent in both conditions.  

To show that CBC3 implies CBC2, let $\eps\in(0,R)\setminus\cN$ and $\omega\in\Sigma_b(\lambda\eps)$. Using the notation $B^\omega(\eps)$ from \eqref{eqn:B-omega}, we observe that
\begin{align}
 C_k^\var\left(F_\eps, (S_\omega F)_\eps \cap B^{\omega,\eps}\right) &\le \sum_{\omega'\in\Sigma (\lambda\eps)\setminus\{\omega\}} C_k^\var\left(F_\eps, (S_\omega F)_\eps \cap (S_{\omega'} F)_\eps\right)\,.
\end{align}
We can restrict the summation to those $\omega'$ for which the intersection $(S_\omega F)_\eps\cap (S_{\omega'} F)_\eps$ is nonempty. By \cite[Lemma~5.3.1, p.45]{winter}, the number of such terms is bounded by some constant $\Gamma_{\max}$ (independent of $\eps$ or $\omega$). (Note that this is where the assumption $\lambda>R\rho^{-1}$ is used.)
Since CBC3 is assumed to hold, each term in this sum is bounded by $a_k\eps^k$, giving the upper bound
$\Gamma_{\max} a_k \eps^k$ for the whole sum. This completes the proof. 
\end{proof}

The following statement establishes that the families $\Sigma_b(\cdot)$, which occur in conditions CBC1 -- CBC3 above, can equivalently be replaced by the larger families $\Sigma(\cdot)$. 

\begin{thm} \label{thm:CBC'}
Each of the following conditions is equivalent to CBC:\vspace{1mm}\\ 
{\bf (CBC1')} There is a constant $c_k$ and a null set $\cN\subset(0,R)$ such that for all $\eps\in(0,R)\setminus\cN$ and all $\sigma\in \Sigma(\eps)$
\begin{equation}\label{eq:CBC1'}
C_k^\var\left(F_\eps, (S_\sigma F)_\eps \cap \bigcup_{\sigma'\in\Sigma (\eps)\setminus\{\sigma\}}(S_{\sigma'}F)_\eps\right)\le c_k\eps^k.
\end{equation}
{\bf (CBC2')} 
There exist $\lambda\ge 1$, $b_k=b_k(\lambda)>0$ and a null set $\cN$ such that for all $\eps\in(0,R/\lambda)\setminus\cN$ and all $\omega\in \Sigma(\lambda\eps)$
\begin{equation*}
C_k^\var\left(F_\eps, (S_\omega F)_\eps \cap \bigcup_{\omega'\in\Sigma (\lambda\eps)\setminus\{\omega\}}(S_{\omega'}F)_\eps\right)\le b_k\eps^k
\end{equation*}
and such that for all $\eps\in[R/\lambda,R)\setminus\cN$ 
\begin{equation*}
C_k^\var\left(F_\eps\right)\le b_k\eps^k.
\end{equation*}
{\bf (CBC3')} There exist $\lambda\ge\max\{1, R\rho^{-1}\}$, $a_k=a_k(\lambda)>0$ and a null set $\cN$ such that for all $\eps\in(0,R/\lambda)\setminus\cN$ and $\omega\omega'\in\Sigma(\lambda \eps)$ with $\omega\neq\omega'$
\begin{equation*}
C_k^\var\left(F_\eps, (S_\omega F)_\eps \cap (S_{\omega'}F)_\eps\right)\le a_k\eps^k
\end{equation*}
and such that for all $\eps\in[R/\lambda,R)\setminus\cN$ 
\begin{equation*}
C_k^\var\left(F_\eps\right)\le a_k\eps^k.
\end{equation*}
\end{thm}

\begin{proof}
The implications CBC1' $\Rightarrow$ CBC1, CBC2' $\Rightarrow$ CBC2 and CBC3' $\Rightarrow$ CBC3 are obvious. The implication CBC1 $\Rightarrow$ CBC2'
holds, since in the proof of CBC1 $\Rightarrow$ CBC2 in Theorem~\ref{thm:CBC2} it is only used that $\omega\in\Sigma(\lambda\eps)$ but not that $\omega\in\Sigma_b(\lambda\eps)$. The proofs of the implications
CBC2' $\Rightarrow$ CBC1' and CBC2' $\Rightarrow$ CBC3' are completely analogous to the proofs of  CBC2 $\Rightarrow$ CBC1 and CBC2 $\Rightarrow$ CBC3 in Theorems~\ref{thm:CBC2} and \ref{thm:CBC3}, respectively. One can replace each instance of $\Sigma_b(\cdot)$ by $\Sigma(\cdot)$ and use the ''stronger`` condition CBC2' instead of CBC2. Thus we have the following cycles of implications: CBC1 $\Rightarrow$ CBC2' $\Rightarrow$ CBC1'$\Rightarrow$ CBC1 and CBC1 $\Rightarrow$ CBC2' $\Rightarrow$ CBC3' $\Rightarrow$ CBC3, which, together with the equivalences in Theorems~\ref{thm:CBC2} and \ref{thm:CBC3}, show the equivalence to CBC of each of these three conditions.
\end{proof}

\begin{cor}
The validity of CBC is independent of the choice of the open set $O$.  
\end{cor}
\begin{proof} 
By Theorem~\ref{thm:CBC'}, CBC is equivalent to CBC1', a condition in which the open set $O$ does not occur.
\end{proof}

We point out that in concrete examples some of these conditions are easier to verify than the original condition. However, we postpone examples to the next section, where a simpler but slightly stronger condition is discussed which is even easier to verify. 

To complete the picture of the present state of the art regarding the curvature bound condition, we briefly discuss an example of a self-similar set not satisfying CBC. It was discovered independently by Andreas Wust and Jan Rataj.  In fact, in the example below we discuss an one-parameter family of sets $F(p)$, $p\in(0,\frac 12)$, for which CBC fails. In the proof we use one of the equivalent reformulations of CBC.   

\begin{figure}
\begin{minipage}{65mm}  
  \includegraphics[width=65mm]{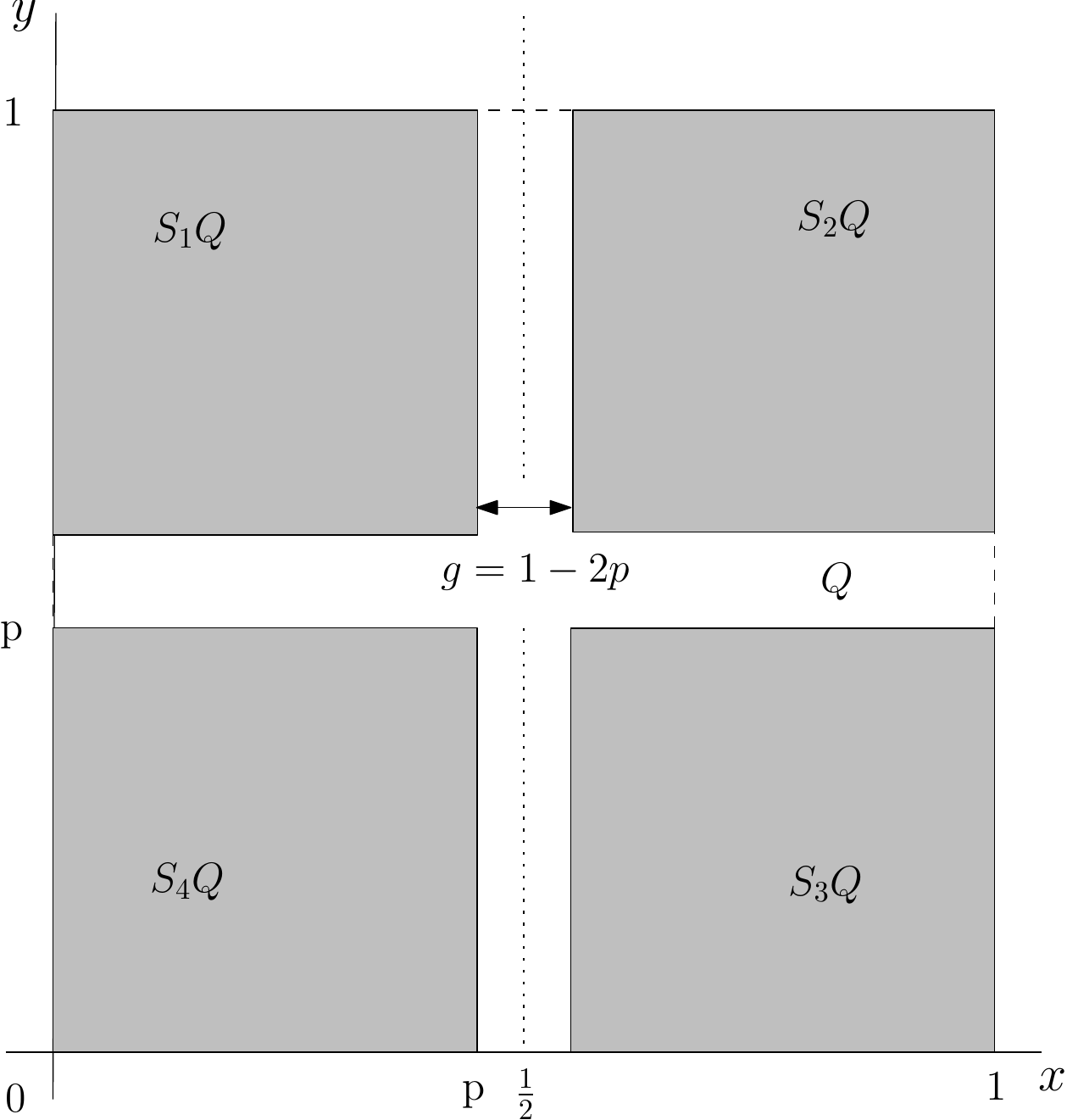}
 \end{minipage}
\begin{minipage}{50mm}
  \includegraphics[width=50mm]{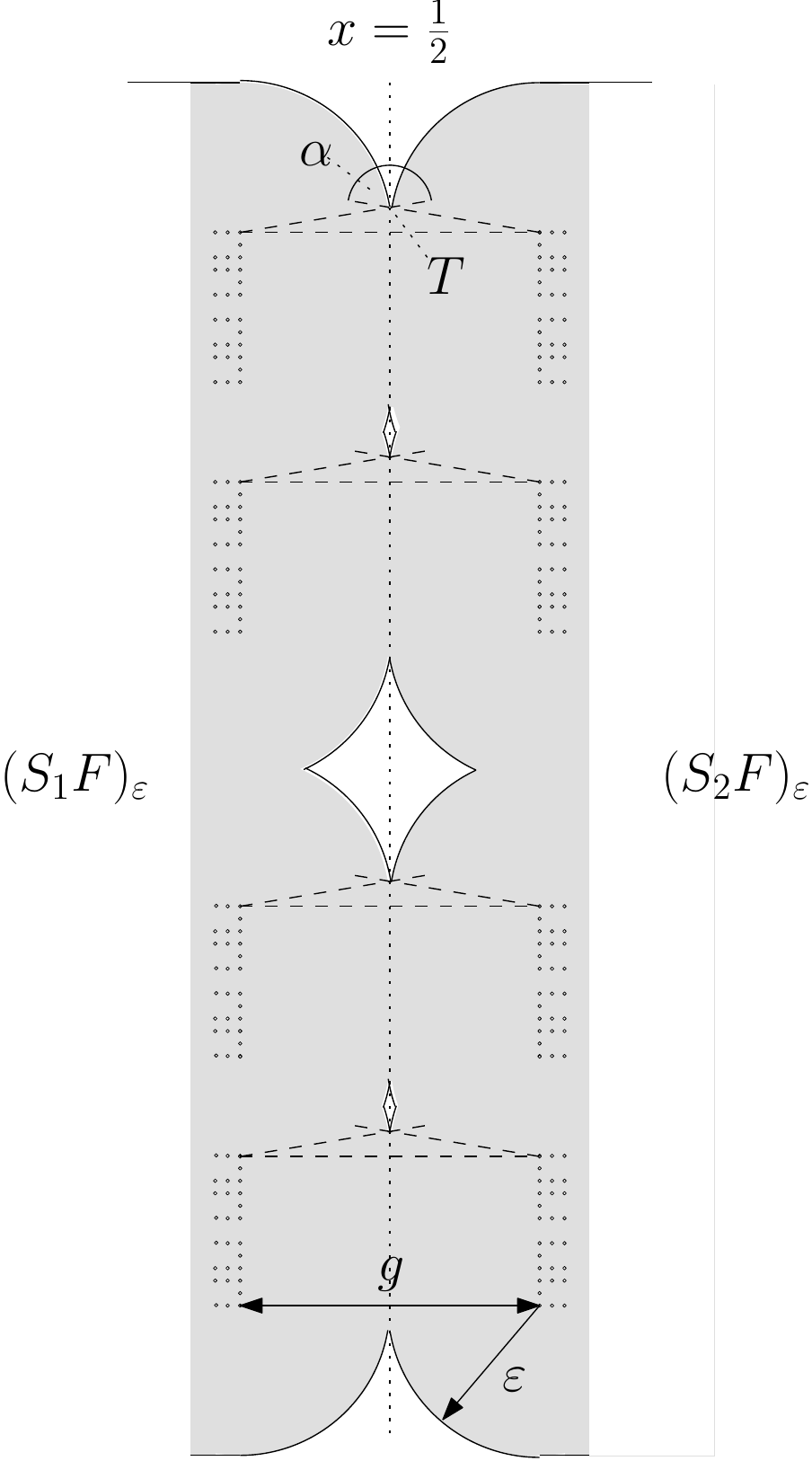}
\end{minipage} \label{fig:wust}
\caption{(Left): the iterates of the unit square $Q=[0,1]^2$ under the IFS generating the set $F(p)$ of Example~\ref{ex:wust}; (Right): enlargement of the intersection $(S_1 F)_\eps\cap (S_2 F)_\eps$ for some $\eps>\frac g2$}
\end{figure}

\begin{ex} \label{ex:wust}
For $p\in(0,\frac 12)$, let $F=F(p)$ be the self-similar set in $\R^2$ generated by the four similarities $S_1,\ldots, S_4$ each with contraction ratio $p$, which map the unit square $Q=[0,1]^2$ to the four squares of side length $p$ in the corners of $Q$, cf.~Figure~\ref{fig:wust}.
$F$ is a Cantor set satisfying the strong separation condition (and thus in particular SOSC). $F$ can also be viewed as the Cartesian product $C\times C$, where $C=C(p)$ is the self-similar Cantor set on $\R$ generated by the two mappings $f_1(x)=px$ and $f_2(x)=px+(1-p)$, $x\in \R$.
It is clear that $\frac g2$ is a critical value of the distance function of $F$, where $g:=1-2p$ is the minimal distance between $S_1 F$ and $S_2 F$. 
Note that for $\eps<\frac g2$, the intersection $(S_1 F)_\eps\cap (S_2 F)_\eps$ is empty, while for $\eps=\frac g2$ it is a Cantor set $\widetilde C$ on the vertical line $x=\frac 12$, which is similar to $C$ (but shrinked by a factor $\frac 13$). For $\eps>\frac g2$, the intersection consists of a finite number of (roughly lense-shaped) connected components whose number increases as $\eps\searrow\frac g2$. The fact, that the number of these components is unbounded as $\eps\searrow\frac g2$, is essentially the reason, why CBC fails.

To provide a rigorous argument, we will now demonstrate that CBC2' (and thus, by Theorem~\ref{thm:CBC'}, CBC) is not satisfied for $F$.  Choose $R$ and $\lambda$ such that $\frac R\lambda = \frac g{2p}$, $R>\sqrt{2}\diam(F)=2$ and $\lambda\ge 1$. (This can for instance be achieved as follows: For $p\ge\frac 18$, choose $R=3$ and $\lambda=R \frac{2p}{1-2p}\ge 1$. For $p<\frac 18$, choose $R=\frac{1-2p}{2p}>3$ and $\lambda=1$.) Let $\omega=1$ and $\omega'=2$. These choices ensure that $\omega,\omega'\in\Sigma(\lambda\eps)$ for each $\eps\in(p\frac R \lambda, \frac R \lambda]=(\frac g2, \frac g{2p}]$.
The validity of CBC2' would in particular imply the existence of a constant $b_0$ and a null set $\cN$ such that for all $\eps\in(p \frac R\lambda,\frac R\lambda)\setminus \cN$  
$$
C_0^\var(F_\eps, (S_1 F)_\eps\cap (S_2 F)_\eps)\le C_0^\var(F_\eps, (S_1 F)_\eps\cap \bigcup_{i\neq 1}(S_i F)_\eps)\le b_0.
$$
Therefore, it suffices to 
show that for each constant $b>0$ there is a set $I=I(b)\subset (p R/\lambda, R/\lambda]$ with $\lambda_1(I)>0$ such that 
for all $\eps\in I$,
$$
C_0^\var(F_\eps, (S_1 F)_\eps\cap (S_2 F)_\eps)\ge b.
$$
Observe that the number $N(\eps)$ of connected components of the set $(S_1 F)_\eps\cap (S_2 F)_\eps$ is given by one plus the number of those complementary intervals $L$ of the set $\widetilde C$, whose length $l$ satisfies $l^2>4\eps^2 - g^2$, cf.\ Figure~\ref{fig:wust}. Each component $K$ of $(S_1 F)_\eps\cap (S_2 F)_\eps$ has exactly 2 points in common with the set $\partial F_\eps$, namely the endpoints of the segment $K\cap \{x=\frac 12\}$. Moreover, by symmetry, the curvature $C_0(F_\eps, \mydot)$ at each of these points is the same (for fixed $\eps$) and given by the angle $\alpha=\alpha(\eps)$ at the point $T$ in Figure~\ref{fig:wust}. 
It is not difficult to see that $\alpha= 2 \arcsin(\frac g{2\eps})$, which implies that $\alpha>\frac g{\eps}$, since $\arcsin(x)>x$ for $x\in(0,1)$.  
Hence we obtain, for $\eps\in (\frac g2, \frac g{2p}]$, that $\alpha>2p$ and thus
$$
C_0^\var(F_\eps, (S_1 F)_\eps\cap (S_2 F)_\eps)= 2 N(\eps) C_0^\var(F_\eps, \{Q\})=2 N(\eps) \frac\alpha {2\pi}>\frac{2p}\pi N(\eps)
$$
So fix some $b>0$. Choose $u>\frac g2$ such that $N(u)\frac{2p}\pi>b$. (This choice is possible, since $N(\eps)\to\infty$ as $\eps\searrow \frac g2$). Let $I:=(\frac g2,u)$. Clearly, $\lambda_1(I)>0$ and, since $N(\mydot)$ is monotone decreasing, we have for all $\eps\in I$, $C_0^\var(F_\eps, (S_1 F)_\eps\cap (S_2 F)_\eps)>\frac{2p}\pi N(\eps)\ge \frac{2p}\pi N(u)>b$ as desired. This shows that CBC fails for each of the sets $F=F(p)$ with $p\in(0,\frac 12)$.
\end{ex}

\section{A simpler but stronger condition} \label{sec:SCBC}

In view of the results in \cite{lw04} and \cite{winter}, it is a natural question to ask, whether the curvature bound condition can also be formulated 
in terms of intersections of first level cylinder sets. Indeed, even formula \eqref{eq:global:lim} in Theorem~\ref{thm:global} suggests this, since the function $R_k$ defined in \eqref{Rk-def} describes essentially the curvature (of $F_\eps$) in the intersections of first level cylinder sets. 
However, it turns out that the condition below  which involves only first level cylinder sets is sufficient but not necessary for CBC to be satisfied. We call this simpler condition the \emph{strong curvature bound condition} (SCBC). It provides a useful tool for the discussion of concrete examples.

\begin{thm} \label{thm:SCBC}
Let $k\in\{0,\ldots,d-2\}$ and let $F$ be a self-similar set in $\R^d$ satisfying OSC and RC. Then the following condition implies CBC:\\
{\bf (SCBC)} There is a constant $d_k$ and a null set $\cN$ such that for all $\eps\in(0,R)\setminus\cN$ and all $i,j\in\{1,\ldots,N\}$ with $i\neq j$, 
\begin{equation*}
C_k^\var\left(F_\eps, (S_i F)_\eps \cap (S_j F)_\eps\right)\le d_k\eps^k.
\end{equation*}
\end{thm}
\begin{proof} Fix some $\lambda\ge\max\{1, R\rho^{-1}\}$. We show that SCBC implies CBC2' (with the same null set $\cN$ and this choice of $\lambda$), which is equivalent to CBC by Theorem~\ref{thm:CBC'}.   

For  $\eps\in(0,R/\lambda)\setminus\cN$ and $\omega\in \Sigma(\lambda\eps)$, consider the family
$$
\Omega:=\{\omega'\in \Sigma(\lambda\eps)\setminus\{\omega\}: (S_\omega F)_\eps\cap (S_{\omega'} F)_\eps\neq\emptyset\}.
$$
By \cite[Lemma~5.3.1, p45]{winter}, the cardinality of $\Omega$ is bounded by some constant $\Gamma_{\max}$ (independent of $\eps$ and $\omega\in\Sigma(\lambda\eps)$), giving an upper bound for the number of terms in the double union below.
Write $m:=|\omega|$, $\omega=\omega_1\omega_2\ldots \omega_m$ and $\omega|n:=\omega_1\ldots\omega_n$ for $n=0,1,\ldots, m$. Observe that 
\begin{align} \label{eqn:CBC2union}
(S_\omega F)_\eps\,&\cap B^{\omega,\eps} = (S_\omega F)_\eps\,\cap \bigcup_{\omega'\in\Omega} (S_{\omega'} F)_\eps\\
&=\bigcup_{n=0}^{m-1} \bigcup_{\stackrel{\omega'\in\Omega}{\omega'|n =\omega|n, \omega'_{n+1}\neq\omega_{n+1}}} \left[(S_\omega F)_\eps \cap (S_{\omega'} F)_\eps \setminus \bigcup_{\stackrel{\sigma\in\Sigma(\lambda \eps)}{\sigma|n\neq \omega|n}} (S_\sigma F)_\eps\right]\notag.
\end{align}
Indeed, for each $\omega'\in\Omega$ there is an unique $n\in\{0,\ldots,m-1\}$ such that $\omega'|n =\omega|n$ but $\omega'_{n+1}\neq\omega_{n+1}$. Moreover, from the intersection $(S_\omega F)_\eps \cap (S_{\omega'} F)_\eps$ we can safely subtract all sets $(S_\sigma F)_\eps$ with $\sigma\in\Sigma(\lambda \eps)$ and $\sigma|n\neq \omega|n$, since either $\sigma\notin\Omega$, in which case $(S_\sigma F)_\eps$ has no intersection with $(S_\omega F)_\eps$ and thus no intersection with $(S_\omega F)_\eps \cap (S_{\omega'} F)_\eps$, or $\sigma\in\Omega$, in which case the set $(S_\sigma F)_\eps$ occurs already in the union for some smaller $n$. 

We infer that 
\begin{align} \label{eqn:CBC2sum}
&C_k^\var (F_\eps, (S_\omega F)_\eps\cap B^{\omega,\eps})\\ \notag
&\le \sum_{n=0}^{m-1} \sum_{\stackrel{\omega'\in\Omega}{\omega'|n =\omega|n, \omega'_{n+1}\neq\omega_{n+1}}} 
C_k^\var\left(F_\eps, (S_\omega F)_\eps \cap (S_{\omega'} F)_\eps \setminus \bigcup_{\stackrel{\sigma\in\Sigma(\lambda \eps)}{\sigma|n\neq \omega|n}} (S_\sigma F)_\eps\right)
\end{align}
where we keep in mind that the number of terms in this double sum is bounded by $\Gamma_{\max}$. 
Furthermore, each term in the double sum is bounded from above as follows.
For fixed $\omega'\in\Omega$ (and the corresponding $n$) write $\tilde\omega:=\omega|n=\omega'|n$. The sets $F_\eps$ and $(S_{\tilde\omega} F)_\eps$ coincide inside the open set 
$$
U:= \left(\bigcup_{\stackrel{\sigma\in\Sigma(\lambda \eps)}{\sigma|n\neq \tilde\omega}} (S_\sigma F)_\eps\right)^c.
$$
Hence, by the locality property \eqref{eqn:loc} and by the scaling properties \eqref{eqn:motion} and \eqref{eqn:scale}, we obtain
\begin{align*}
C_k^\var&\left(F_\eps, (S_\omega F)_\eps \cap (S_{\omega'} F)_\eps \cap U\right)\\
&= C_k^\var\left((S_{\tilde\omega} F)_\eps, (S_\omega F)_\eps \cap (S_{\omega'} F)_\eps \cap U\right)\\
&\le C_k^\var\left((S_{\tilde\omega} F)_\eps, (S_\omega F)_\eps \cap (S_{\omega'} F)_\eps\right)\\
&\le C_k^\var\left(S_{\tilde\omega} F_{\eps/r_{\tilde\omega}}, S_{\tilde\omega}\left((S_{\omega_{n+1}} F)_{\eps/r_{\tilde\omega}} \cap (S_{\omega'_{n+1}} F)_{\eps/r_{\tilde\omega}}\right)\right)\\
&= r_{\tilde\sigma}^k C_k^\var\left(F_{\eps/r_{\tilde\omega}}, (S_{\omega_{n+1}} F)_{\eps/r_{\tilde\omega}} \cap (S_{\omega'_{n+1}} F)_{\eps/r_{\tilde\omega}}\right).
\end{align*}
Applying now SCBC, we conclude that the last term is bounded by $d_k \eps^k$ and thus the whole expression in \eqref{eqn:CBC2sum} by
$b_k\eps^k$, where $b_k:=\Gamma_{\max} d_k$. Since this bound is valid for all $\eps\in(0,R/\lambda)\setminus\cN$ and $\omega\in \Sigma(\lambda\eps)$, the proof of the first inequality of CBC2' is complete.  

For the second inequality of CBC2', we decompose the set $F_\eps$ as follows
$$
F_\eps=\bigcup_{\sigma\in\Sigma(\eps)} (S_\sigma F)_\eps = \bigcup_{\sigma\in\Sigma(\eps)} \left(\left((S_\sigma F)_\eps\cap A^{\sigma,\eps}\right) \cup\left( (S_\sigma F)_\eps\cap (A^{\sigma,\eps})^c\right)\right).
$$  
For $\eps\in(R/\lambda, R)$, the cardinality of $\Sigma(\eps)$ is uniformly bounded by the constant $\hat c:=\#\Sigma(R/\lambda)$. Therefore, it suffices to show that there is a constant $c$ such that the curvature of each set in this union is bounded by $c \eps^k$.
For the sets $(S_\sigma F)_\eps\cap (A^{\sigma,\eps})^c$, one can use directly \eqref{eqn:loc} (in the open set $(A^{\sigma,\eps})^c$) to infer that
\begin{align*}
C_k^\var\left(F_\eps, (S_\sigma F)_\eps\cap (A^{\sigma,\eps})^c\right)&=C_k^\var\left((S_\sigma F)_\eps, (S_\sigma F)_\eps\cap (A^{\sigma,\eps})^c\right)\\
&\le C_k^\var\left((S_\sigma F)_\eps\right)\le r_\sigma^k C_k^\var\left(F_{\eps/r_\sigma}\right).
\end{align*}
Since $\sigma\in\Sigma(\eps)$ and thus $\eps/r_\sigma>R$, we conclude from \eqref{eq:R-big}, that the last expression is bounded by $c(F,R)\eps^k$ as desired. For the sets $(S_\sigma F)_\eps\cap A^{\sigma,\eps}$ a similar argument as for the sets $(S_\omega F)_\eps\,\cap B^{\omega,\eps}$ in \eqref{eqn:CBC2union} works. One has the decomposition 
$$
  (S_\sigma F)_\eps\cap A^{\sigma,\eps}=\bigcup_{n=0}^{|\sigma|-1} \bigcup_{\stackrel{\sigma'\in\Sigma(\eps)\setminus\{\sigma\}}{\sigma'|n =\sigma|n, \sigma'_{n+1}\neq\sigma_{n+1}}} \left((S_\sigma F)_\eps \cap (S_{\sigma'} F)_\eps \setminus \bigcup_{\stackrel{\tau\in\Sigma(\eps)}{\tau|n\neq \sigma|n}} (S_\tau F)_\eps\right).
$$
Again the number of sets in this double union is bounded, but for a different reason as before. Here the cardinality of $\Sigma(\eps)$ is bounded by $\hat c$ (since $\eps>R/\lambda$). The remaining arguments carry over from the case  $(S_\omega F)_\eps\,\cap B^{\omega,\eps}$ and one obtains the bound
$\hat c d_k \eps^k$ for $C_k^\var\left(F_\eps,(S_\sigma F)_\eps\cap A^{\sigma,\eps}\right)$. This completes the proof of the second inequality of CBC2'.
\end{proof}

We will now show that the converse of Theorem~\ref{thm:SCBC} is not true, i.e., that SCBC is not equivalent to CBC, by providing a counterexample. We will discuss a set which satisfies CBC but not SCBC.

\begin{ex}\label{ex:Uset} (U-set) Consider the self-similar set $F\subset\R^2$ generated by the seven similarities $S_1, \ldots, S_7$, each with ratio $r=\frac 13$, mapping the unit square $Q:=[0,1]^2$ to one of the seven subsquares forming the set $U$ as depicted in Figure~\ref{fig:Uset1}. (Note that $S_4$ includes a clockwise rotation by $\frac{\pi}2$.) This modification of the Sierpinski carpet is similar to the U-sets discussed in \cite{lw04} and \cite{winter}, but in contrast to those sets, the present set $F$ does not have polyconvex parallel sets. For instance, for large $\eps$, the intersection of $F_\eps$ with the upper half space $y\ge 1$ cannot be represented as a finite union of convex sets.

\begin{figure}
\begin{minipage}{6cm}  
  \includegraphics[width=6cm]{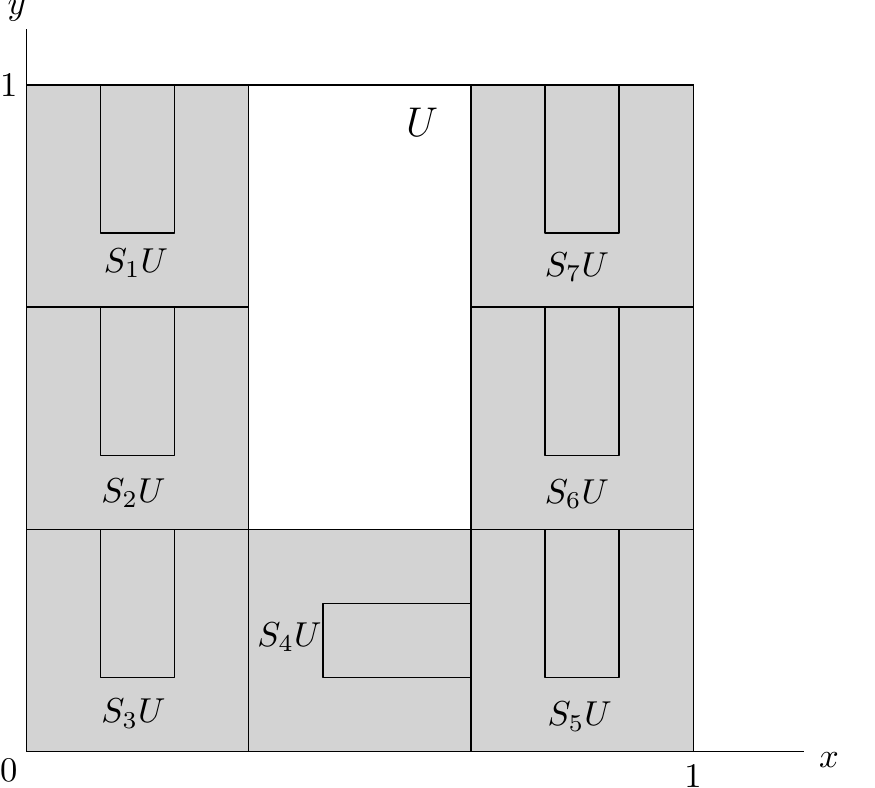}
 \end{minipage}
\begin{minipage}{65mm}
  \includegraphics[width=65mm]{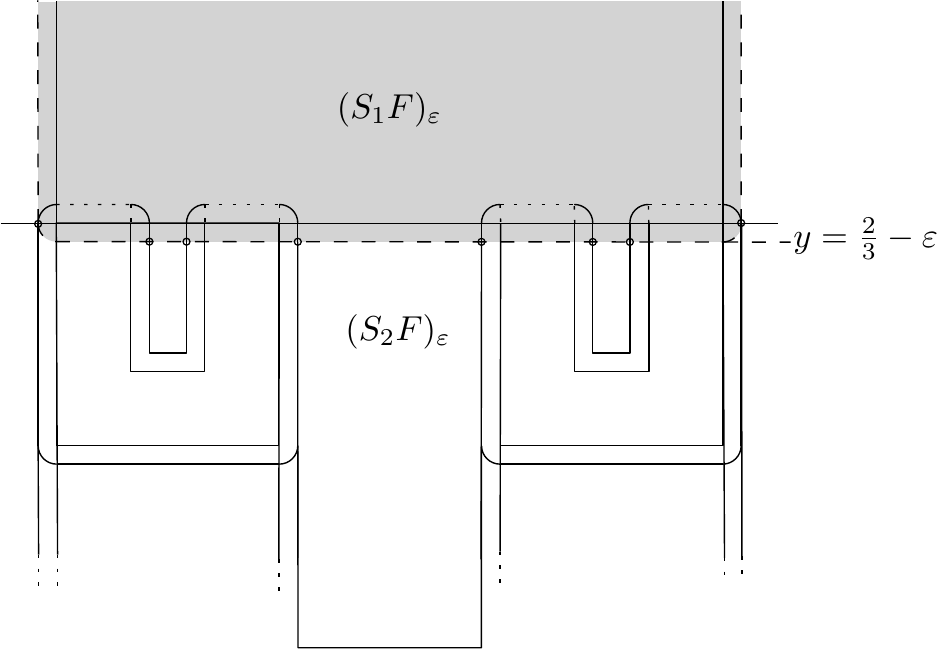}
   
\end{minipage} \label{fig:Uset1}\label{fig:Uset2}
\caption{(Left):  the iterates of the set $U$ under the IFS in Example~\ref{ex:Uset}; (Right): enlargement of the intersection $(S_1 F)_\eps\cap (S_2 F)_\eps$ for $\eps=\frac 1{45}$}
\end{figure}

First we look at the measure $C_0(F_\eps,\mydot)$ at the intersection $(S_1 F)_\eps\cap (S_2 F)_\eps$. We will show that
for $\eps\in[\frac 12 3^{-(m+2)},\frac 12 3^{-(m+1)})$ and $m=1,2,\ldots$
\begin{align} \label{eq:exSCBC}
C_0^\var(F_\eps, (S_1 F)_\eps\cap (S_2 F)_\eps)&=  \frac 12 (2^m-1)\,,  
\end{align}
which immediately implies that  $C_0^\var(F_\eps, (S_1 F)_\eps\cap (S_2 F)_\eps)\to\infty$ as $\eps\to 0$. Hence this curvature cannot be bounded by a constant on the whole interval $(0,R)$ and so SCBC does not hold.

For a proof of \eqref{eq:exSCBC}, observe that the intersection $S_1 F\cap S_2 F$ is a scaled copy $\widetilde C$ of the usual middle-third Cantor set (scaled by a factor $\frac 13$) on the line $y=\frac 23$. Moreover, for $\eps\in(0,R)$, the intersection $\partial F_\eps \cap (S_1 F)_\eps \cap (S_2 F)_\eps$ consists of a finite number of pairs of points on the line $y=\frac23-\eps$, where each pair corresponds to a complementary interval of $\widetilde C$ of length greater than $2\eps$, cf. Figure~\ref{fig:Uset2}. (There are two more intersection points with coordinates $(-\eps,\frac 23)$ and $(\frac 13+\eps,\frac 23)$, which carry no curvature.) In $\widetilde C$, we have one complementary interval of length $\frac 19$, two of length $\frac 1{3^3}$, four of length $\frac 1{3^3}$ and so on, i.e.,\ $2^k$ of length $\frac 1{3^{k+2}}$ for $k=0,1,\ldots$. Therefore, the number $J(\eps)$ of complementary intervals of $\widetilde C$ with length greater than $2\eps$ is given by 
$$
J(\eps)=\sum_{k=0}^{m-1} 2^k=2^{m}-1
$$
for $2\eps\in[3^{-(m+2)},3^{-(m+1)})$ and $m=1,2,\ldots$. Since each of the points contributes a curvature of $-\frac 14$ to $C_0(F_\eps,\mydot)$, we obtain the result claimed in \eqref{eq:exSCBC}. This completes the proof of the assertion that SCBC is not satisfied. 

It remains to show that, on the contrary, CBC is satisfied.
We demonstrate this by verifying CBC2' for $F$. For this purpose, fix $R>2$ and choose $\lambda\ge 1$ large enough to ensure that, for any $\omega,\omega' \in\Sigma(\lambda\eps)$, the intersection $(S_\omega F)_\eps\cap (S_{\omega'} F)_\eps$ is nonempty only if the intersection $S_\omega F\cap S_{\omega'} F$ is, i.e., only if the cylinder sets $S_\omega F$, $S_{\omega'} F$ are direct neighbors. (Any choice $\lambda\ge 6R$ works. Two cylinder sets $\omega,\omega'\in\Sigma(\lambda \eps)$, which do not intersect each other, have distance at least $r^{|\omega|}$ as there is a square of this side length between them. On the other hand, $\omega,\omega'\in\Sigma(\lambda \eps)$ implies $\lambda\eps\le R r^{|\omega|-1}$, i.e. $2\eps< r^{|\omega|}$.) Obviously, a cylinder set $S_\omega F$ can have at most eight neighbors (corresponding to the eight neighboring squares). In fact, it can have at most 5 neighbors, since there are always at least three neighboring squares whose interior is outside $F$ and which do not contain any cylinder set of $F$, but we will not use this. 
To verify the first inequality of CBC2', it suffices to show that there is a constant $b>0$ such that for $\eps\in(0,R/\lambda)$ and $\omega,\omega'\in\Sigma(\lambda\eps)$ with $\omega\neq \omega'$, 
\begin{align}\label{eq:U1}
C_0^\var(F_\eps, (S_\omega F)_\eps\cap (S_\omega' F)_\eps)\le b,
\end{align}
since this clearly implies that $C_0^\var(F_\eps, (S_\omega F)_\eps\cap B^{w,\eps})$ is bounded by $8b$.
So fix $\eps\in(0,R/\lambda)$ and $\omega,\omega'\in\Sigma(\lambda\eps)$ with $\omega\neq \omega'$. Then the intersection $(S_\omega F)_\eps\cap (S_{\omega'} F)_\eps$ is a scaled copy of one of the following four sets: $K_1:=(S_1 F)_\delta \cap (S_2 F)_\delta$, $K_2:=(S_3 F)_\delta\cap (S_4 F)_\delta$, $K_3:=(S_2 F)_\delta\cap (S_4 F)_\delta$ or $K_4:=(S_4 F)_\delta\cap (S_6 F)_\delta$ where $\delta:=\eps 3^{|\omega|-1}$. Moreover, the intersection of $(S_\omega F)_\eps\cap (S_{\omega'} F)_\eps$ with $\partial F_\eps$ is a scaled copy of the corresponding intersection $K_i\cap\partial F_\delta$. This implies 
$$
C_0^\var(F_\eps, (S_\omega F)_\eps\cap (S_\omega' F)_\eps)\le\max_{i\in\{1,2,3,4\}} C_0^\var(F_\delta, K_i)
$$
For $i=2,3,4$, it is easily seen that the set $\partial F_\delta\cap K_i$ consists of 2 points (for all $\delta>0$) and thus $C_0^\var(F_\delta, K_i)$ is certainly bounded by $2$. For $i=1$, we infer that $\delta\ge\frac R\lambda r$ (since $\omega,\omega'\in\Sigma(\lambda \eps)$) and thus $\delta> \frac 13 \frac 2{12}=\frac 1{18}$ by the choice of $R$ and $\lambda$. Hence $K_1$ is connected and so $\partial F_\delta\cap K_i$ consists of 2 points as in the other cases. Therefore the maximium above is clearly bounded by $2$, which completes the proof of \eqref{eq:U1} and thus of the first inequality of CBC2'.

It remains to provide a proof of the second inequality of CBC2'. With the choice $\lambda=6R$ above, it remains to show that 
$C_0^\var(F_\eps)$ is bounded by some constant for $\eps\in (\frac 16,R)$. It is easy to see that $F_\eps$ and the parallel set $Q_\eps$ of the unit square $Q=:[0,1]^2$ coincide in the open half plane $H:=\{(x,y)\in\R^2:y<1\}$. Hence, by \eqref{eqn:loc},
$$
C_0^\var(F_\eps, H)= C_0^\var(Q_\eps,H)\le C_0^\var(Q_\eps)=1,
$$
where the last equality is due to the convexity of $Q_\eps$.
It remains to show that for some $\eta>0$ and $H^\eta:=\{(x,y): y\ge 1-\eta\}$
we also have 
$$
C_0^\var(F_\eps, H^\eta)\le b
$$
for $\eps\in(\frac 16, R)$.

Fix $\eta<\frac{1}6$. Let $\Omega=\{1,7\}^2$ and $A^\eps:=\bigcup_{\omega\in\Omega} (S_\omega F)_\eps$. Observe that for $\frac 16 <\eps$, 
$$
F_\eps\cap H^\eta =A^\eps \cap H^\eta.
$$
Since the diameter of each of the cylinder sets $S_\omega F$ in $A^\eps$ is $\sqrt{2}r^2$, we can infer from \eqref{eq:R-big}, that 
$C_0^\var((S_\omega F)_\eps)$ is bounded by some constant $c=c(R')$ for all $\eps\ge R':=Rr^2$ (and all $\omega\in\Omega$). Therefore,
$$
C_0^\var(F_\eps, H^\eta) =  C_0^\var(A^\eps, H^\eta)\le C_0^\var(A^\eps)
\le \sum_{\sigma\in\Omega} C_0^\var(A^\eps, (S_\sigma F)_\eps\setminus B)+ C_0^\var(A^\eps, B),
$$
where $B:=\bigcup_{\omega,\omega'\in\Omega, \omega\neq \omega'} (S_\omega F)_\eps\cap (S_{\omega'} F)_\eps$.
Since, by \eqref{eqn:loc}, 
$$
C_0^\var(A^\eps,(S_\sigma F)_\eps\setminus B)=C_0^\var((S_\sigma F)_\eps, (S_\sigma F)_\eps\setminus B)\le C_0^\var((S_\sigma F)_\eps)\le c\,,
$$
we infer that the sum above is bounded by $4c$. For the last term observe that     
$$ 
C_0^\var(A^\eps, B)\le \sum_{\omega,\omega'\in\Omega, \omega\neq \omega'} C_0^\var(A^\eps, (S_\omega F)_\eps\cap (S_{\omega'} F)_\eps).
$$
By noting that each of the intersections $(S_\omega F)_\eps\cap (S_{\omega'} F)_\eps$ above is a convex set (or empty) and that the intersection with $\partial A^\eps$ consists of just two points (or none) each contributing at most $\frac 12$ to the curvature of $A^\eps$, we conclude that each term in the latter sum is bounded by $1$ and thus the whole sum by $6$. This completes the proof of the second inequality in CBC2'.
\end{ex}

We conclude this section with a discussion of the well-known Koch curve. As its parallel sets are clearly not polyconvex, it provides an example of a self-similar set to which the results in \cite{Za09} and \cite{wz10} apply but which is not covered by the results in \cite{winter}. It also illustrates how SCBC simplifies the verification of CBC (compare with Example~\ref{ex:Uset}).

\begin{ex} \label{ex:Koch} (Koch curve)
Let $K\subset\R^2\cong\C$ be the self-similar set generated by the two similarity mappings $S_1, S_2$ given (in complex coordinates) by $S_1(z)=c\bar{z}$ and $S_2(z)=(1-c)(\bar{z}-1)+1$, respectively, where $c=\frac 12+i\frac{\sqrt{3}}{6}$. The contraction ratios are $r_1=r_2=r=\frac 1{\sqrt{3}}$. It is well known (and easily seen) that $K$ satisfies OSC.

The critical values of the distance function are $\frac 19 r^k$, $k=0,1,2,\ldots$. In particular, these values form a null set so that RC is satisfied. (More precisely, all critical points lie either on the axis $\Re(z)=\frac12$ or on one of its iterates $S_\omega(\{\Re(z)=\frac12\})$, $\omega\in\Sigma^*$. For $\eps=\frac 1 9$, for instance, $p=\frac 12+ i \frac{\sqrt{3}}{18}$ is the unique critical point with this distance from $K$ (cf.\ Figure~\ref{fig:koch3}). 
Note that also for the critical values $\eps$ the curvature measure $C_0(K_\eps,\mydot)$ is well defined in this case.)

Now we want to look more closely at the curvature bound condition for $k=0$. We will verify that SCBC holds, which implies CBC by Theorem~\ref{thm:SCBC}. Hence instead of having to work with cylinder sets of all levels, it is enough to look at the first level cylinder sets, of which there are only two in this case. It suffices to show that, for all $\eps>(0,R)$,
the expression
$
C_0^\var(K_\eps, (S_1 K)_\eps\cap (S_2 K)_\eps)
$
is bounded by some constant $d_0$. Since the measure $C_0(K_\eps,\mydot)$ is concentrated on the boundary of $K_\eps$, it is enough to consider the intersection  $\partial(S_1 K)_\eps\cap \partial(S_2 K)_\eps$, which consists of some arc $A=A(\eps)$ of the circle of radius $\eps$ centered at the intersection point of $S_1K$ and $S_2K$ and a single point $p=p(\eps)$, the intersection point of the two curves bounding the parallel sets $(S_1K)_\eps$ and $(S_2K)_\eps$ from below, compare Figure~\ref{fig:koch3}. (In fact, it requires some justification to see that those two curves intersect in a single point for each fixed $\eps>0$. We skip the details of the rather elementary computations at this point.) Now observe that $C_0(K_\eps, A)=\frac{\alpha}{2\pi}$ where $\alpha=\frac \pi 3$ is the angle determining the arc $A$ (independent of $\eps$). Hence, $C_0(K_\eps, A)=\frac{1}{6}$ for all $\eps>0$.
The measure $C_0(K_\eps, \{p\})$ depends on $\eps$ and is negative. It is certainly bounded from below by $-1$. (In fact, it is bounded by $-\frac12$.) Hence, we get
that, for all $\eps>0$,
$$
C_0^\var(K_\eps, (S_1 K)_\eps\cap (S_2 K)_\eps)\le C_0^\var(K_\eps, A\cup \{p\})\le \frac 16 + 1 =: d_0,
$$
which verifies SCBC and thus CBC.

Since $K$ is lattice, Theorem~\ref{thm:global} implies the existence of the average limit $C_0^f(K)$, as given by \eqref{eq:global:lim}, but not the existence of the essential limit in \eqref{eq:global:esslim}. Moreover, by Theorem~\ref{thm:local}, the corresponding fractal curvature measure $C_0^f(K,\mydot)$ exists and is given by $C_0^f(K) \mu_K$, where $\mu_K=\frac{\Ha^D\lfloor K(\mydot)}{\Ha^D(K)}$ is the normalized $D$-dimensional Hausdorff measure on $K$ with $D=\log_3 4$.

\begin{figure}
  \centering
  \scalebox{0.6}{\includegraphics{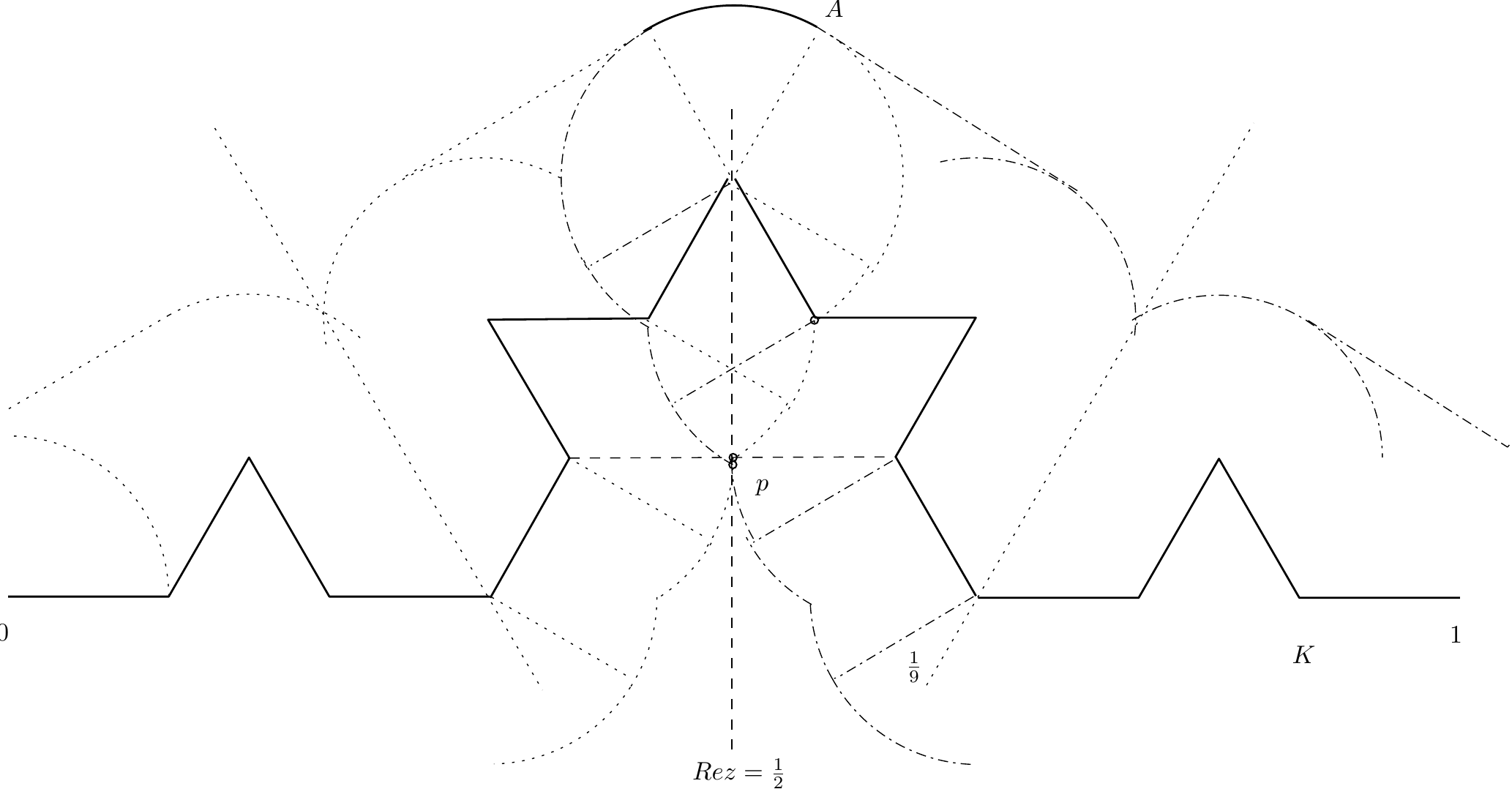}}
  \caption{Approximation of the Koch curve and its $\eps$-parallel set for the critical value $\eps=\frac 19$. $p$ is a critical point realizing this value. The intersection  $\partial(S_1 K)_\eps\cap \partial(S_2 K)_\eps$ consists of the arc $A$ and the point $p$.}
  \label{fig:koch3}
\end{figure}
\end{ex}

\end{document}